\documentclass[]{theclass}
\usepackage{mathtools}
\usepackage{comment}
\definecolor{orchid}{rgb}{0.85, 0.44, 0.84}

\begin{document}

\begin{frontmatter}

\titledata{On $2$-bisections and monochromatic edges in claw-free cubic multigraphs}{}

\authordata{Federico Romaniello}
{Dipartimento di Matematica, Informatica ed Economia\\ Universit\`{a} degli Studi della Basilicata, Potenza, Italy}
{federico.romaniello@unibas.it}{}


\begin{abstract}
A $k$-bisection of a multigraph $G$ is a partition of its vertex set into two parts of the same cardinality such that every connected component induced by the vertices in each part has at most $k$ vertices. Cui and Liu shown that every claw-free cubic multigraph contains a $2$-bisection, while Eom and Ozeki constructed specific $2$-bisections with bounded number of monochromatic edges. Their bound is the best possible for claw-free cubic simple graphs. In this note, we extend the latter result to the larger family of claw-free cubic multigraphs. 
\end{abstract}
\end{frontmatter}
\section{Introduction}
In this short note, we will mostly use multigraphs, i.e. graphs that may contain multiple edges but no loops are allowed. A graph that does not contain multiple edges is said to be \textit{simple}. A \textit{bisection} $(B,W)$ of a (multi)graph $G$ is a partition of its vertex set $V(G)$ in two sets, say $B$ and $W$, such that $|B|=|W|$. A $k$-bisection of $G$ is a bisection $(B,W)$ such that every component of $G[B]$ and $G[W]$ has at most $k$ vertices, where $G[X]$ denotes the subgraph of $G$ induced by the subset $X \subseteq V(G)$. A \textit{monochromatic edge} is an edge connecting two vertices of the same part of a given bisection. A \textit{claw} is the complete bipartite graph $K_{1,3}$, and the vertex of degree three in such a claw is said to be its \textit{centre}. If a (multi)graph does not contain induced claws it is said to be \textit{claw-free}. For basic notation on (multi)graphs not defined here, we refer the reader to \cite{BM}.
Bisections of cubic simple graphs and multigraphs have already been considered by many researchers. The following is an important open conjecture, due to Ban and Linial \cite{banlinial}:
\begin{conjecture}[\cite{banlinial}]
Every bridgeless cubic simple graph $G$ admits a $2$-bisection, unless $G$ is the Petersen graph.
\end{conjecture}
This conjecture has been proved true for bridgeless claw-free cubic simple graphs in \cite{clawfree}; however, in \cite{cuiliu} a stronger result has been proved by Cui and Liu:
\begin{theorem}[\cite{cuiliu}]
Every claw-free cubic multigraph contains a $2$-bisection.
\end{theorem}
Moreover, in \cite{emt} a more general conjecture has been presented and suggests a $2$-bisection in every cubic simple graph which admits a perfect matching, except for the Petersen graph. Furthermore, the same paper highlights a significant connection between bisections and nowhere-zero flows, due to a result by Jaeger \cite{Jae} showing that a cubic simple graph with a circular nowhere-zero $r$-flow contains an $(\left\lfloor r \right\rfloor - 2)$-bisection.

It is worth noting that a (multi)graph contains a $1$-bisection if and only if it is bipartite. To some extent, $2$-bisections of a (multi)graph can be seen as an index of how close it is to being bipartite, as its monochromatic components are either isolated vertices or independent edges. By considering this, Eom and Ozeki directed their attention to the count of monochromatic edges in $2$-bisections of a claw-free cubic simple graph $G$ \cite{kenta}. The idea is that a smaller number of monochromatic edges in a $2$-bisection indicates a closer proximity of $G$ to be bipartite. It is worth remarking that their bounds are the best for all claw-free cubic simple graphs \cite[Section 1]{kenta}. Our contribution is to refine the results presented in \cite{kenta} by extending them to the case of claw-free cubic multigraphs (see Theorem \ref{mainthm}).\\
\section{Preliminaries}
We will deal solely with cubic multigraphs, i.e. $3$-regular multigraphs. Given a $2$-bisection $(B,W)$, we denote by $\varepsilon(B,W)$ the number of its monochromatic edges. Some induced subgraphs, shown in Figure \ref{sub}, will play an important role in our proofs:
\begin{itemize}
    \item  a \textit{digon} is a subgraph on two vertices consisting of a pair of parallel edges, i.e. a cycle of length two;
    \item a \textit{triangle} is a subgraph consisting of a cycle of length three;
    \item the Shannon multigraph $Sh(3)$, i.e. a particular triangle with a digon on one of its sides. Due to its characteristic shape we will refer to $Sh(3)$ as a \textit{trumpet};
    \item a \textit{diamond} is a subgraph isomorphic to the complete graph on four vertices with an edge deleted, i.e. $K_4 - e$.
\end{itemize} 
\begin{figure}[h!]
    \centering
    \includegraphics[width=.6\textwidth]{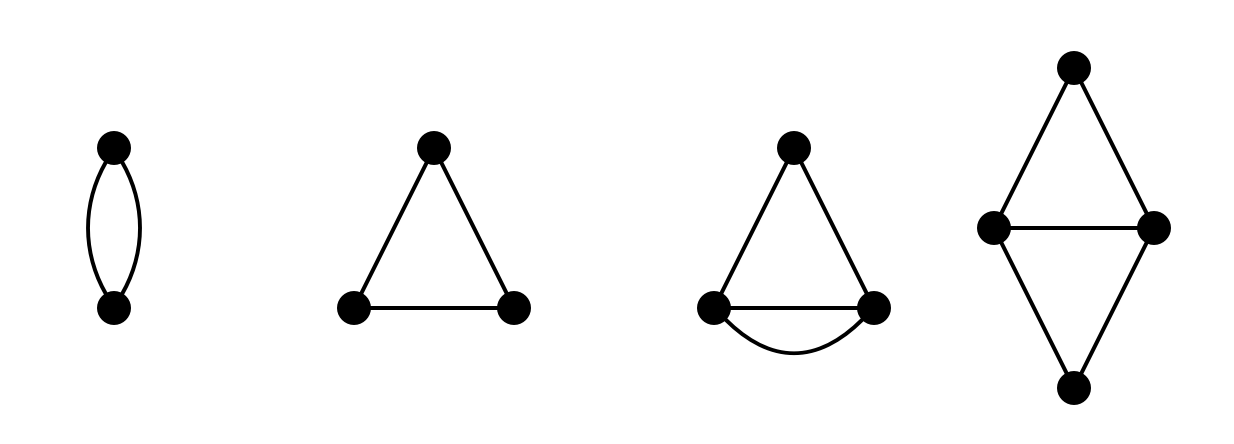} 
    \caption{From left to right: a digon, a triangle, a trumpet and a diamond.}
    \label{sub}
\end{figure}

It is easy to check that any two diamonds in a connected claw-free cubic multigraph $G$ are vertex-disjoint, unless $G \cong K_4$; hence, for cubic multigraphs, the claw-free condition is equivalent to requiring that every vertex should belong to a triangle, a trumpet or a digon, since a diamond is made of two triangles that share a side.\\
In particular, in any cubic simple graph or multigraph, the maximum number of triangles in which a vertex may lie is three, and this can only occur in a component isomorphic to $K_4$. On the other hand, a vertex lies in exactly two triangles if it is one of the vertices of degree three in a diamond, or one of the vertices of the multiple edges in a trumpet.

The following statement can be proved by a simple counting argument.
\begin{lemma}\label{epsbespw}
Let $(B,W)$ be a $2$-bisection of a cubic multigraph $G$ and let $\varepsilon_B$, $\varepsilon_W$ be the number of monochromatic edges in $G[B]$ and $G[W]$ respectively. Then $\varepsilon_B=\varepsilon_W$.
\end{lemma}

In the proof of Theorem \ref{mainthm} we will construct a bisection with specific conditions, highlighted in the next definition:
\begin{definition}
Let $G$ be a connected claw-free cubic multigraph, different from $K_4$. A bisection $(B,W)$ of $G$ is said to be \textit{desired}, if the following conditions hold:
\begin{itemize}
    \item[\textbf{(DB1)}] exactly one edge in every triangle is monochromatic;
    \item[\textbf{(DB2)}] every monochromatic edge is contained in a triangle;
    \item[\textbf{(DB3)}] every diamond contains exactly one monochromatic edge;
    \item[\textbf{(DB4)}] for two vertices, say $u$, $v$, joining multiple edges, the edges between $u$ and $v$ are not monochromatic.
\end{itemize}
\end{definition}
It follows that, in a desired bisection, the multiple edge of a trumpet is not monochromatic.

Moreover, we will also avail of the next two results, that have been proved in \cite{kenta}.

\begin{lemma}[\cite{kenta}]\label{des2}
A desired bisection in a connected claw-free cubic multigraph is a $2$-bisection.
\end{lemma}
\begin{theorem}[\cite{kenta}]\label{keven}
Let $G$ be a connected claw-free cubic multigraph that is not $K_4$. If the number $k$ of diamonds in $G$ is even, then $G$ admits a desired bisection.
\end{theorem}
\section{Main Theorem}\label{main}
As stated in the introduction, this section shall be devoted to stating and proving our main result. We would like to remark that this is a natural generalisation of Theorem $1.2$ and Proposition $1.3$ stated and proved in \cite[Section 1]{kenta}.
\begin{theorem}\label{mainthm}
 Let $G$ be a connected claw-free cubic multigraph, $G \not\cong K_4$; let $k$ be the number of diamonds, and let $p$ be the number of digons in $G$. Then $G$ contains a $2$-bisection $(\bar{B},\bar{W})$ such that:
 \begin{equation*}
\varepsilon(\bar{B},\bar{W}) = 
\left\{
\begin{alignedat}{2}
&\frac{|V(G)|-k-2p}{3}     &&,\text{ if $k$ is even}\\[6pt]
&\frac{|V(G)|-k-2p}{3}+1 \ &&,\text{ if $k$ is odd.}
\end{alignedat}
\right.
\end{equation*}

Moreover,  $\varepsilon(\bar{B},\bar{W}) \leq \varepsilon(B,W)$, for every $2$-bisection $(B,W)$ of $G$. 
\end{theorem}
\begin{proof}
Let $G$ be a connected claw-free cubic multigraph, different from $K_4$. We start by determining the minimum number of monochromatic edges for any $2$-bisection of $G$.\\
Let $(B,W)$ be a $2$-bisection of $G$. $V(G)$ can be partitioned into vertex-disjoint diamonds, triangles, trumpets, and digons, as $G$ is claw-free and different from $K_4$. Moreover, each triangle of $G$ contains at least one monochromatic edge, and it is easy to check that:
\begin{equation*}
    \varepsilon(B,W) \geq k + \frac{|V(G)|-4k-2p}{3}= \frac{|V(G)|-k-2p}{3}. 
\end{equation*}
This is the minimum number of monochromatic edges when the number of diamonds $k$ is even.\\
Suppose now the number of diamonds $k$ is odd and $\varepsilon(B,W)= \frac{|V(G)|-k-2p}{3}$. By Lemma \ref{epsbespw}, $\varepsilon(B,W)=\varepsilon_B +\varepsilon_W=2\varepsilon_B=2\varepsilon_W$ is an even integer. This is not possible as it contradicts the fact that $|V(G)|$ is even and the assumption that $k$ is odd. Hence, it must be that $\varepsilon(B,W)> \frac{|V(G)|-k-2p}{3}$. Moreover, from the decomposition of $V(G)$ into vertex-disjoint diamonds, triangles, trumpets, and digons, it follows that $\frac{|V(G)|-k-2p}{3}$ must be an integer. In conclusion, the minimum number of monochromatic edges satisfies $\varepsilon(B,W)\geq \frac{|V(G)|-k-2p}{3}+1$, when $k$ is odd.\\
We now construct a specific $2$-bisection with the sought number of monochromatic edges. We shall distinguish between two cases, whether $k$ is even or odd.\\

\noindent\textbf{Case $k$ even:} Lemma \ref{des2} and Theorem \ref{keven} guarantee that $G$ contains a desired $2$-bisection, say $(\bar{B},\bar{W})$. Conditions \textbf{(DB1)}, \textbf{(DB2)}, \textbf{(DB3)}, and \textbf{(DB4)}, and our assumptions, imply that $V(G)$ can be partitioned into vertex-disjoint diamonds, triangles, trumpets, and digons. We have that:
\begin{equation}
    \varepsilon(\bar{B},\bar{W})= k + \frac{|V(G)|-4k-2p}{3}= \frac{|V(G)|-k-2p}{3},
\end{equation}
where, in the first equality, the term $k$ corresponds to an edge in each diamond, while the second term accounts for an edge in each triangle or trumpet. Hence the thesis.\\

\noindent\textbf{Case $k$ odd:} Since $k$ is odd, $G$ contains at least one diamond. Let $a,b,c,d$ be the four vertices with $ad \not\in E(G)$ and let $x$ and $y$ be the neighbours of $a$ and $d$ not in the diamond, respectively. Note that $x\neq y$, otherwise $G$ would contain a claw. It follows that $\left\{xa,dy\right\}$ is an edge-cut of $G$.  Let $G^{\prime}$ be the multigraph obtained from $G$ by deleting the four vertices $a,b,c,d$ and adding a new edge that connects $x$ and $y$, see Figure \ref{odddm}.\\
We will now show some structural properties of $G^{\prime}$. It is easy to check that $G^{\prime}$ is a connected and cubic multigraph. Suppose that $G^{\prime}$ contains an induced claw. Because $G$ is claw-free, the only possibility for $G^{\prime}$ to contain a claw is by using the new edge $xy$, and this implies that the centre of the claw is either $x$ or $y$. By symmetry, without loss of generality, we may assume $x$. Then, by using $a$ instead of $y$ we would obtain an induced claw in $G$, which is a contradiction. Hence, $G^{\prime}$ is a connected, cubic, claw-free multigraph.\\
We will show now that the new edge $xy \in V(G^{\prime})$ is not contained in any diamonds of $G^{\prime}$. Suppose firstly that the new edge $xy \in E(G)$, i.e. it was already an edge of $G$. It could be that $xy$ was a multiedge in $G$. In this case $G^{\prime}$ must be the cubic multigraph on two vertices and three multiedges, because of the claw-freeness of $G$ and $G^{\prime}$, see Figure \ref{fig:trivial}. Suppose now that $xy \in E(G)$ is a simple edge, and let $x^{\prime}$ and $y^{\prime}$ be the third neighbours of $x$ and $y$, respectively, other than $a$ and $d$. If $x^{\prime} \neq y^{\prime}$, $G$ would have a claw, which is a contradiction. Hence $x^{\prime} = y^{\prime} =w$. The claw-freeness of $G$ and $G^{\prime}$ implies that $xy$ is the parallel edge in the trumpet $wxy$ in $G^{\prime}$, see Figure \ref{fig:mtrump}.
\begin{figure}
\centering
\begin{subfigure}{0.5\textwidth}
  \centering
  \includegraphics[width=.8\linewidth]{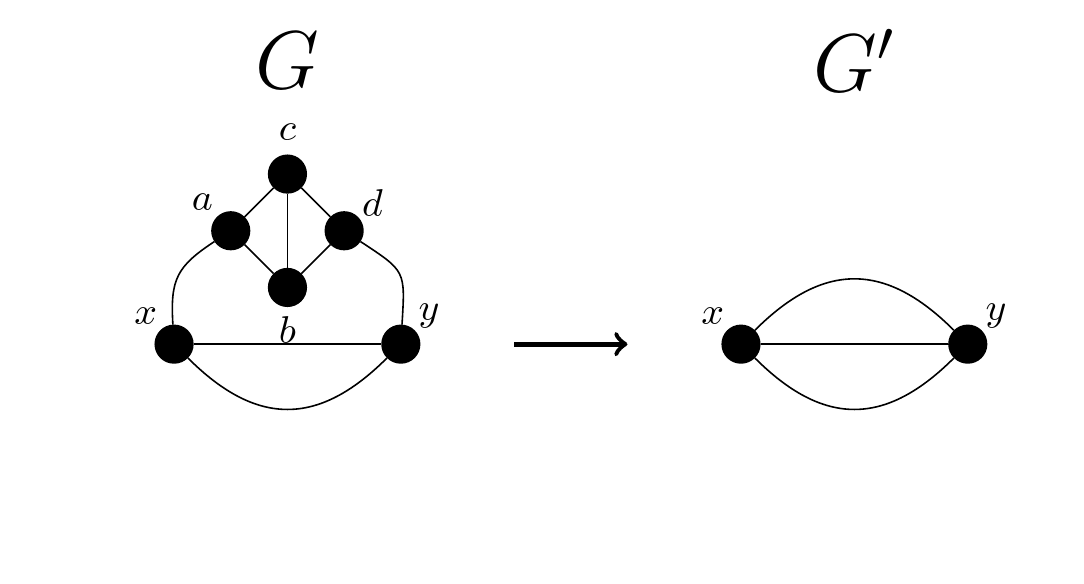}
  \caption{~}
  \label{fig:trivial}
\end{subfigure}%
\begin{subfigure}{0.5\textwidth}
  \centering
  \includegraphics[width=.8\linewidth]{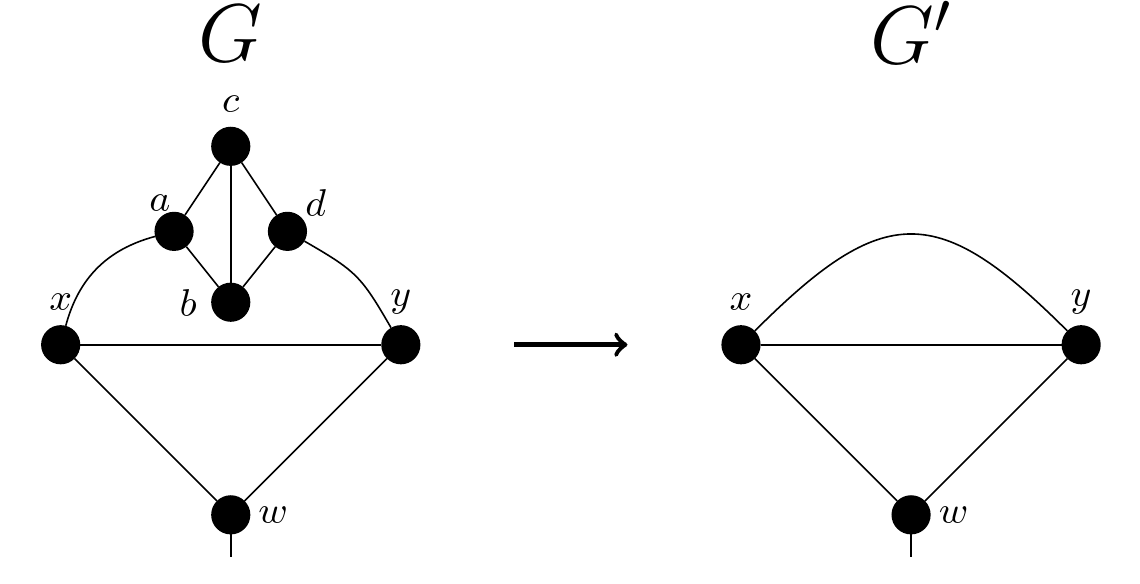}
  \caption{~}
  \label{fig:mtrump}
\end{subfigure}
\caption{The two possible cases when the new edge $xy$ was already an edge of $G$.}
\label{fig:test}
\end{figure}
   
Suppose now that $xy \not\in E(G)$ and the new edge $xy \in G^{\prime}$ is contained in a triangle. Let $z$ be the third vertex of such a triangle. Without loss of generality, we can assume that there are no multiple edges in $G^{\prime}$ connecting $x$ and $z$. If this were the case, there must be no multiple edges connecting $y$ and $z$, and therefore it is possible to swap the roles of $x$ and $y$. Let $\bar{x}$ be the neighbour of $x$ in $G^{\prime}$ other than $y$ and $z$. We must have $\bar{x}z \in E(G)$ to forbid the presence of an induced claw with centre $x$ in $G$. We remark that the neighbours of $z$ are precisely $x,y$ and $\bar{x}$. If $y\bar{x} \not\in E(G)$ then in the graph $G$ there would be an induced claw with centre $y$, which is a contradiction. Hence, $y\bar{x} \in E(G)$, which would imply that $G^{\prime} \cong K_4$, another contradiction as $G$ would have exactly eight vertices and $G$ would be a ring of diamonds with two diamonds while we are assuming that $k$ is odd. In conclusion, this means that the new edge $xy$ is contained in no triangles of $G^{\prime}$. Specifically, it follows directly that the new edge $xy$ is not contained in any diamonds of $G^{\prime}$.\\
Consequently, the number of diamonds in $G^{\prime}$ is $k-1$, which is an even integer. Let $(B^{\prime},W^{\prime})$ be a desired bisection in $G^{\prime}$, which exists by Theorem \ref{keven}. As above:

\begin{equation*}
\begin{aligned}
    \varepsilon(B^{\prime},W^{\prime}) = \frac{|V(G^{\prime})|-(k-1)-2p}{3}&=\frac{(|V(G)|-4)-(k-1)-2p}{3}\\
    &=\frac{|V(G)|-k-2p}{3}-1.
\end{aligned}
\end{equation*}    

\begin{figure}[h!]
    \centering
    \includegraphics[width=.6\textwidth]{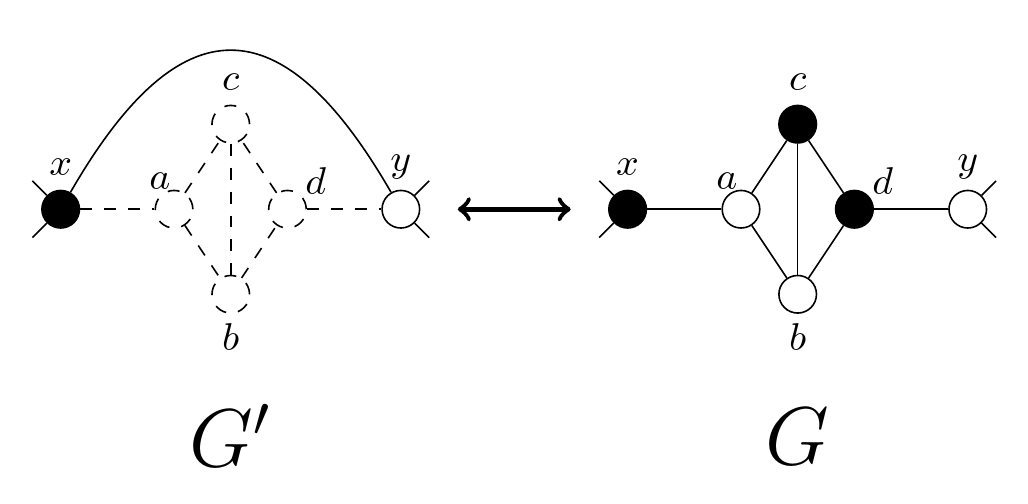} 
    \caption{The $2$-bisections $(B^{\prime},W^{\prime})$ of $G^{\prime}$ and $(B,W)$ of $G$, when the number of diamonds $k$ is odd. }
    \label{odddm}
\end{figure}
We have shown that the new edge $xy \in E(G^{\prime})$ is not contained in any diamonds, or it is a multiple edge. Conditions \textbf{(DB2)} and \textbf{(DB4)} imply that $xy$ is not monochromatic. Without loss of generality, we may assume $x\in B^{\prime}$ and $y\in W^{\prime}$. Let $\bar{B}=B^{\prime} \cup \left\{c,d\right\}$ and $\bar{W}=W^{\prime} \cup \left\{a,b\right\}$, see Figure \ref{odddm}. Therefore,
\begin{equation*}
\varepsilon(\bar{B},\bar{W})=\varepsilon(B^{\prime},W^{\prime})+2 = \frac{|V(G)|-k-2p}{3}+1.
\end{equation*}
\end{proof}

\section*{Acknowledgements}
The author would like to thank Kenta Ozeki for useful discussions about the topic. Moreover, we gratefully acknowledge the anonymous referees for their valuable comments and suggestions which significantly improved the quality of this manuscript.

\end{document}